\documentclass[reqno]{amsart}\usepackage{amsmath}

\usepackage{amsfonts}
\usepackage{amssymb}
\usepackage{hyperref}

\begin{document}
\title[ Positive solutions ]{positive solutions of nonlinear three-point integral boundary-value problems for second-order differential equations}
\author[F. Haddouchi, S. Benaicha]{Faouzi Haddouchi, Slimane Benaicha}
\address{Faouzi Haddouchi\\
Department of Physics, University of Sciences and Technology of
Oran, El Mnaouar, BP 1505, 31000 Oran, Algeria}
\email{fhaddouchi@gmail.com}
\address{Slimane Benaicha \\
Department of Mathematics, University of Oran, Es-senia, 31000 Oran,
Algeria} \email{slimanebenaicha@yahoo.fr}

\begin{abstract}
We investigate the existence of positive solutions to the nonlinear
second-order three-point integral boundary value problem
\begin{equation*} \label{eq-1}
\begin{gathered}
{u^{\prime \prime }}(t)+a(t)f(u(t))=0,\ 0<t<T, \\
u(0)={\beta}u(\eta),\ u(T)={\alpha}\int_{0}^{\eta}u(s)ds,
\end{gathered}
\end{equation*}
where $0<{\eta}<T$, $0<{\alpha}< \frac{2T}{{\eta}^{2}}$,
$0\leq{\beta}<\frac{2T-\alpha\eta^{2}}{\alpha\eta^{2}-2\eta+2T}$ are
given constants. We show the existence of at least one positive
solution if $f$ is either superlinear or sublinear by applying
Krasnoselskii's  fixed point theorem in cones.
\end{abstract}

\maketitle \numberwithin{equation}{section}
\newtheorem{theorem}{Theorem}[section]
\newtheorem{lemma}[theorem]{Lemma} \newtheorem{proposition}[theorem]{%
Proposition} \newtheorem{corollary}[theorem]{Corollary} \newtheorem{remark}[%
theorem]{Remark}
\newtheorem{exmp}{Example}[section]

\section{Introduction}
The study of the existence of solutions of multipoint boundary value
problems for linear second-order ordinary differential equations was
initiated by II'in and Moiseev \cite{Mois}. Then Gupta \cite{Gupt}
studied three-point boundary value problems for nonlinear
second-order ordinary differential equations. Since then, nonlinear
second-order three-point boundary value problems have also been
studied by several authors. We refer the reader to \cite{Chen, Guo,
Han, Liang1, Liang2, Li, Liu1, Liu2, Liu3, Luo, Ma1, Ma2, Ma3, Pang,
Sun1, Sun2, Xu} and the references therein.

Tariboon and Sitthiwirattham \cite{Tarib} proved the
existence of positive solutions for the three-point boundary-value
problem with integral condition

\begin{equation} \label{eq-2}
{u^{\prime \prime }}(t)+a(t)f(u(t))=0,\ t\in(0,1),
\end{equation}
\begin{equation} \label{eq-3}
u(0)=0,\ u(1)={\alpha}\int_{0}^{\eta}u(s)ds,
\end{equation}
where $0<{\eta}<1$ and $0<{\alpha}< \frac{2}{{\eta}^{2}}$.

This paper is concerned with the existence of positive solutions of
the equation

\begin{equation} \label{eq-4}
{u^{\prime \prime }}(t)+a(t)f(u(t))=0,\ t\in(0,T),
\end{equation}
with the three-point integral boundary condition
\begin{equation} \label{eq-5}
u(0)={\beta}u(\eta),\ u(T)={\alpha}\int_{0}^{\eta}u(s)ds,
\end{equation}
where ${\alpha}>0$, ${\beta}\geq0$, ${\eta}\in (0,T)$ are given
constants. Clearly if ${\beta}=0$ and $T=1$, then \eqref{eq-5}
reduces to \eqref{eq-3}. The purpose of this paper is to give some
results for existence for positive solutions to
\eqref{eq-4}-\eqref{eq-5}, assuming that $0<{\alpha}<
\frac{2T}{{\eta}^{2}}$,
$0\leq{\beta}<\frac{2T-\alpha\eta^{2}}{\alpha\eta^{2}-2\eta+2T}$ and
$f$ is either superlinear or sublinear.

Our results extend and complete those obtained by J. Tariboon and T.
Sitthiwirattham \cite{Tarib}. On the other hand, we point out that the proof of the last part in the sublinear case $(f_{\infty}=0)$ of Theorem 3.1 in \cite{Tarib} is not correct since it is based on an inequality which is not true. We give a new proof, which is different from that of Theorem 3.1 in \cite{Tarib}, and obtain an extended result.

Set
\begin{equation} \label{eq-6}
f_{0}=\lim_{u\rightarrow0^{+}}\frac{f(u)}{u}, \
f_{\infty}=\lim_{u\rightarrow\infty}\frac{f(u)}{u}.
\end{equation}
Then $f_{0}=0$ and $f_{\infty}=\infty$ correspond to the superlinear
case, and $f_{0}=\infty$ and $f_{\infty}=0$ correspond to the
sublinear case. By the positive solution of
\eqref{eq-4}-\eqref{eq-5} we mean that function $u(t)$ is positive
on $0<t<T$ and satisfies the problem \eqref{eq-4}-\eqref{eq-5}.

Throughout this paper, we assume the following hypotheses:

\begin{itemize}
\item[(H1)]
 $f\in C([0,\infty),[0,\infty))$.

\item[(H2)] $a\in C([0,T],[0,\infty))$ and there exists $t_{0}\in[\eta,T]$ such that
$a(t_{0})>0$.
\end{itemize}

The following theorem (Krasnoselskii's fixed-point theorem), will
play an important role in the proof of our main results.

\begin{theorem} [\cite{Krasn}]\label{theo 1.1}
Let $E$ be a Banach space, and let $K\subset E$ be a cone. Assume
$\Omega_{1}$, $\Omega _{2}$ are open bounded subsets of $E$ with $0\in \Omega _{1}$%
, $\overline{\Omega }_{1}\subset \Omega _{2}$, and let

\begin{equation*}\label{eq-7}
A: K\cap (\overline{\Omega }_{2}\backslash  \Omega
_{1})\longrightarrow K
\end{equation*}

be a completely continuous operator such that either

(i) $\ \left\Vert Au\right\Vert \leq \left\Vert u\right\Vert $, $\
u\in K\cap \partial \Omega _{1}$, \ and $\left\Vert Au\right\Vert
\geq \left\Vert u\right\Vert $, $\ u\in K\cap \partial \Omega _{2}$;
or

(ii) $\left\Vert Au\right\Vert \geq \left\Vert u\right\Vert $, $\
u\in K\cap
\partial \Omega _{1}$, \ and $\left\Vert Au\right\Vert \leq \left\Vert
u\right\Vert $, $\ u\in K\cap \partial \Omega _{2}$

hold. Then $A$ has a fixed point in $K\cap
(\overline{\Omega}_{2}\backslash $ $\Omega _{1})$.
\end{theorem}

\section{Preliminaries}
To prove the main existence results we will employ several
straightforward lemmas. These lemmas are based on the linear
boundary-value problem.

\begin{lemma}\label{lem 2.1}

Let $\beta \neq \frac{2T-\alpha \eta ^{2}}{\alpha \eta ^{2}-2\eta
+2T}$. Then for $y\in C([0,T],\mathbb{R})$, the problem
\begin{equation}\label{eq-8}
{u^{\prime \prime }}(t)+y(t)=0, \ t\in (0,T),
\end{equation}
\begin{equation}\label{eq-9}
u(0)=\beta u(\eta ), \ u(T)=\alpha \int_{0}^{\eta }u(s)ds,
\end{equation}
has a unique solution
\begin{eqnarray*}
u(t)&=&\frac{\beta (2T-\alpha \eta ^{2})-2\beta (1-\alpha \eta
)t}{(\alpha \eta ^{2}-2T)-\beta (2\eta -\alpha \eta
^{2}-2T)}\int_{0}^{\eta }(\eta -s)y(s)ds \\
&&+\frac{\alpha \beta
\eta -\alpha (\beta -1)t}{(\alpha \eta ^{2}-2T)-\beta (2\eta -\alpha
\eta
^{2}-2T)}\int_{0}^{\eta }(\eta -s)^{2}y(s)ds \\
&&+\frac{2(\beta-1)t-2\beta \eta }{(\alpha \eta
^{2}-2T)-\beta (2\eta -\alpha \eta ^{2}-2T)}\int_{0}^{T}(T-s)y(s)ds-%
\int_{0}^{t}(t-s)y(s)ds.
\end{eqnarray*}
\end{lemma}
\begin{proof}
From \eqref{eq-8}, we have
\begin{equation}\label{eq-10}
u(t)=u(0)+u^{\prime }(0)t-\int_{0}^{t}(t-s)y(s)ds.
\end{equation}
Integrating \eqref{eq-10} from $0$ to $\eta $, where $\eta \in
(0,T)$, we have
\begin{eqnarray*}
\int_{0}^{\eta }u(s)ds&=&u(0)\eta +u^{\prime }(0)\frac{\eta ^{2}}{2}%
-\int_{0}^{\eta }(\int_{0}^{\tau }(\tau -s)y(s)ds)d\tau \\
&=&u(0)\eta +u^{\prime }(0)%
\frac{\eta ^{2}}{2}-\frac{1}{2}\int_{0}^{\eta }(\eta -s)^{2}y(s)ds.
\end{eqnarray*}
Since
\begin{equation*}\label{eq-11}
u(T)=u(0)+u^{\prime }(0)T-\int_{0}^{T}(T-s)y(s)ds  \ \  \text{and}
\end{equation*}
\begin{equation*}\label{eq-12}
u(\eta )=u(0)+u^{\prime }(0)\eta -\int_{0}^{\eta }(\eta -s)y(s)ds.
\end{equation*}
By \eqref{eq-9}, from $u(0)=\beta u(\eta )$, we have
\begin{equation*}\label{eq-13}
(\beta -1)u(0)+\eta \beta u^{\prime }(0)=\beta \int_{0}^{\eta }(\eta
-s)y(s)ds,
\end{equation*}
and from $\ u(T)=\alpha \int_{0}^{\eta }u(s)ds$, we have
\begin{equation*}\label{eq-14}
(1-\alpha \eta )u(0)+(T-\alpha \frac{\eta ^{2}}{2})u^{\prime
}(0)=\int_{0}^{T}(T-s)y(s)ds-\frac{\alpha }{2}\int_{0}^{\eta }(\eta
-s)^{2}y(s)ds.
\end{equation*}
Therefore,
\begin{eqnarray*}
u(0)&=&\frac{\beta (2T-\alpha \eta ^{2})}{(\alpha \eta ^{2}-2T)-\beta (2\eta -\alpha \eta ^{2}-2T)}%
\int_{0}^{\eta }(\eta-s)y(s)ds \\
&&-\frac{2\beta \eta}{(\alpha \eta ^{2}-2T)-\beta (2\eta -\alpha \eta ^{2}-2T)} \int_{0}^{T}(T-s)y(s)ds \\
&&+\frac{\alpha \beta \eta}{(\alpha \eta ^{2}-2T)-\beta (2\eta
-\alpha \eta ^{2}-2T)} \int_{0}^{\eta }(\eta -s)^{2}y(s)ds,
\end{eqnarray*}
\begin{eqnarray*}
u^{\prime }(0)&=&\frac{2(\beta -1)}{(\alpha \eta ^{2}-2T)-\beta
(2\eta -\alpha \eta ^{2}-2T)}\int_{0}^{T}(T-s)y(s)ds \\
&&-\frac{\alpha (\beta -1)}{(\alpha \eta ^{2}-2T)-\beta (2\eta -\alpha \eta ^{2}-2T)}\int_{0}^{\eta }(\eta -s)^{2}y(s)ds \\
&&-\frac{2\beta (1-\alpha \eta )}{(\alpha \eta ^{2}-2T)-\beta (2\eta
-\alpha \eta ^{2}-2T)}\int_{0}^{\eta }(\eta -s)y(s)ds,
\end{eqnarray*}

from which it follows that
\begin{eqnarray*}
u(t)&=&\frac{\beta (2T-\alpha \eta ^{2})-2\beta (1-\alpha \eta
)t}{(\alpha \eta ^{2}-2T)-\beta (2\eta -\alpha \eta
^{2}-2T)}\int_{0}^{\eta }(\eta -s)y(s)ds \\
&&+\frac{\alpha \beta \eta -\alpha (\beta -1)t}{(\alpha \eta
^{2}-2T)-\beta (2\eta -\alpha \eta
^{2}-2T)}\int_{0}^{\eta }(\eta -s)^{2}y(s)ds \\
&&+\frac{2(\beta-1)t-2\beta \eta }{(\alpha \eta
^{2}-2T)-\beta (2\eta -\alpha \eta ^{2}-2T)}\int_{0}^{T}(T-s)y(s)ds-%
\int_{0}^{t}(t-s)y(s)ds.
\end{eqnarray*}
\end{proof}

\begin{lemma}\label{lem 2.2}
Let $0<\alpha <\frac{2T}{\eta ^{2}}$, $\ 0\leq \beta <%
\frac{2T-\alpha \eta ^{2}}{\alpha \eta ^{2}-2\eta +2T}$. If $y\in C([0,T],[0,\infty
))$, then the unique solution $u$ of
\eqref{eq-8}-\eqref{eq-9} satisfies $\ u(t)\geq 0$ for $t\in [0,T]$.
\end{lemma}
\begin{proof}
From the fact that $u^{^{\prime \prime }}(t)=-y(t)\leq 0$, we have
that the graph of $u(t)$ is concave down on $(0,T)$ and
\begin{equation} \label{eq-11}
\int_{0}^{\eta }u(s)ds\geq \frac{\eta }{2}(u(0)+u(\eta )).
\end{equation}
Combining \eqref{eq-9} with \eqref{eq-11}, we get
\begin{equation} \label{eq-12}
u(T)\geq \frac{\alpha (\beta +1)\eta }{2}u(\eta ).
\end{equation}
Since the graph of $u$ is concave down, we get
\begin{equation*} \label{eq-13}
\frac{u(\eta )-u(0)}{\eta }\geq \frac{u(T)-u(0)}{T}.
\end{equation*}
Combining this with \eqref{eq-9} and \eqref{eq-12}, \ we obtain
\begin{equation*} \label{eq-14}
(1-\beta )\frac{u(\eta )}{\eta }\geq \frac{\alpha (\beta +1)\eta -2\beta }{%
2T}u(\eta ).
\end{equation*}

If $u(0)<0$, then $u(\eta )<0$. It implies $\beta \geq
\frac{2T-\alpha \eta^{2}}{\alpha \eta ^{2}-2\eta +2T}$, a contradiction to $\beta <\frac{%
2T-\alpha \eta ^{2}}{\alpha \eta ^{2}-2\eta +2T}$.

If $u(T)<0$, then  $u(\eta )<0$, and the same contradiction emerges.
Thus, it is true that $u(0)\geq 0$, $\ u(T)\geq 0,$ together with
the concavity of $u(t)$, we have $u(t)\geq 0$ for \ $t\in [0,T].$
\end{proof}

\begin{lemma}\label{lem 2.3}
Let $\ \alpha >\frac{2T}{\eta ^{2}}$, $\beta\geq0$. If $y\in C([0,T],[0,\infty
))$, then the problem \eqref{eq-8}-\eqref{eq-9} has no positive
solutions.
\end{lemma}
\begin{proof}
Suppose that problem \eqref{eq-8}-\eqref{eq-9} has a positive
solution $u$ satisfying $u(t)\geq 0,$ $t\in [0,T]$.

If $u(T)>0$, then $\int_{0}^{\eta }u(s)ds>0$. It implies
\begin{equation*} \label{eq-15}
u(T)=\alpha \int_{0}^{\eta }u(s)ds>\frac{2T}{\eta ^{2}}\frac{\eta }{2}%
(u(0)+u(\eta ))=\frac{T(\beta +1)u(\eta )}{\eta }\geq \frac{Tu(\eta
)}{\eta },
\end{equation*}
that is
\begin{equation*} \label{eq-16}
\frac{u(T)}{T}>\frac{u(\eta )}{\eta },
\end{equation*}
which is a contradiction to the concavity of $u$.

If  $u(T)=0$, then $\int_{0}^{\eta }u(s)ds=0$, this is $u(t)\equiv 0$ for all $%
t\in [0,\eta ].$ If there exists  $t_{0}\in (\eta ,T)$ such that $%
u(t_{0})>0$, then  $\ u(0)=u(\eta )<u(t_{0})$, a violation of the
concavity of $u$. Therefore, no positive solutions exist.
\end{proof}

\begin{lemma}\label{lem 2.4}
Let $0<\alpha <\frac{2T}{\eta ^{2}}$, $\ 0\leq\beta <%
\frac{2T-\alpha \eta ^{2}}{\alpha \eta ^{2}-2\eta +2T}$. If $y\in
C([0,T],[0,\infty ))$, then the unique solution $u$ of
\eqref{eq-8}-\eqref{eq-9} satisfies
\begin{equation} \label{eq-17}
\min_{t\in [\eta,T]}u(t)\geq \gamma \|u\|,\ \|u\|=\max_{t\in
[0,T]}|u(t)|,
\end{equation}
where
\begin{equation} \label{eq-18}
\gamma:=\min\left\{\frac{\eta}{T},
\frac{\alpha(\beta+1)\eta^{2}}{2T},
\frac{\alpha(\beta+1)\eta(T-\eta)}{2T-\alpha(\beta+1)\eta^{2}}\right\}\in\left(0,1\right).
\end{equation}
\end{lemma}

\begin{proof}
We divide the proof into three cases. Set $u(t_{1})=\|u\|$.

Case 1. If $\eta\leq\ t_{1}\leq T$ and $\min_{t\in
[\eta,T]}u(t)=u(\eta)$, then the concavity of $u$ implies that
\begin{equation*}\label{eq-19}
\frac{u(\eta)}{\eta}\geq\frac{u(t_{1})}{t_{1}}\geq
\frac{u(t_{1})}{T}.
\end{equation*}
Thus,
\begin{equation*}\label{eq-20}
\min_{t\in [\eta,T]}u(t)\geq\frac{\eta}{T}\|u\|.
\end{equation*}

Case 2. If $\eta\leq\ t_{1}\leq T$ and $\min_{t\in
[\eta,T]}u(t)=u(T)$, then \eqref{eq-9}, \eqref{eq-11} and the
concavity of $u$ implies
\begin{eqnarray*}
u(T)&=& \alpha \int_{0}^{\eta }u(s)ds\geq \frac{\alpha(\beta+1)\eta^{2}}{2} \frac{u(\eta)}{\eta}\\
&\geq& \frac{\alpha(\beta+1)\eta^{2}}{2} \frac{u(t_{1})}{t_{1}}\\
&\geq& \frac{\alpha(\beta+1)\eta^{2}}{2}\frac{u(t_{1})}{T}.
\end{eqnarray*}
This implies that
\begin{equation*}\label{eq-21}
\min_{t\in [\eta,T]}u(t)\geq
\frac{\alpha(\beta+1)\eta^{2}}{2T}\|u\|.
\end{equation*}

Case 3. If $t_{1}\leq\eta< T$, then $\min_{t\in [\eta,T]}u(t)=u(T)$.
Using the concavity of $u$ and  \eqref{eq-9}, \eqref{eq-11}, we
obtain
\begin{eqnarray*}
u(t_{1})&\leq&u(T)+\frac{u(T)-u(\eta)}{T-\eta}(t_{1}-T)\\
&\leq& u(T)+\frac{u(T)-u(\eta)}{T-\eta}(0-T)\\
&\leq&u(T)\left[1-T\frac{1-\frac{2}{\alpha(\beta+1)\eta}}{T-\eta}\right]\\
&=&\frac{2T-\alpha(\beta+1)\eta^{2}}{\alpha(\beta+1)\eta(T-\eta)}u(T),
\end{eqnarray*}
from which it follows that
\begin{equation*}\label{eq-20}
\min_{t\in[\eta,T]}u(t)\geq
\frac{\alpha(\beta+1)\eta(T-\eta)}{2T-\alpha(\beta+1)\eta^{2}}\|u\|.
\end{equation*}
Summing up, we have
\begin{equation*}\label{eq-21}
\min_{t\in [\eta,T]}u(t)\geq \gamma \|u\|,
\end{equation*}
where
\begin{equation*}\label{eq-27}
\gamma:=\min\left\{\frac{\eta}{T},
\frac{\alpha(\beta+1)\eta^{2}}{2T},
\frac{\alpha(\beta+1)\eta(T-\eta)}{2T-\alpha(\beta+1)\eta^{2}}\right\}.
\end{equation*}
This completes the proof.
\end{proof}

\section{Existence of positive solutions}
Now we are in the position to establish the main result.
\begin{theorem}\label{theo 3.1}
Assume {\rm (H1)} and {\rm (H2)} hold, and\ $0<\alpha <\frac{2T}{\eta ^{2}}$, $\ 0\leq\beta <%
\frac{2T-\alpha \eta ^{2}}{\alpha \eta ^{2}-2\eta +2T}$. Then the
problem \eqref{eq-4}-\eqref{eq-5} has at least one positive solution
in the case

(i) $f_{0}=0$ and $f_{\infty}=\infty$ (superlinear), or

(ii) $f_{0}=\infty$ and $f_{\infty}=0$ (sublinear).
\end{theorem}

\begin{proof}
It is known that $0<\alpha <\frac{2T}{\eta ^{2}}$, $\ 0\leq\beta <%
\frac{2T-\alpha \eta ^{2}}{\alpha \eta ^{2}-2\eta +2T}$. From Lemma
\ref{lem 2.1}, $u$ is a solution to the boundary value problem
\eqref{eq-4}-\eqref{eq-5} if and only if $u$ is a fixed point of
operator $A$, where  $A$ is defined by
\begin{eqnarray*}
Au(t)&=&\frac{\beta (2T-\alpha \eta ^{2})-2\beta (1-\alpha \eta
)t}{(\alpha \eta ^{2}-2T)-\beta (2\eta -\alpha \eta
^{2}-2T)}\int_{0}^{\eta }(\eta -s)a(s)f(u(s))ds \\
&&+\frac{\alpha \beta \eta -\alpha (\beta -1)t}{(\alpha \eta
^{2}-2T)-\beta (2\eta -\alpha \eta
^{2}-2T)}\int_{0}^{\eta }(\eta -s)^{2}a(s)f(u(s))ds \\
&&+\frac{2(\beta-1)t-2\beta \eta }{(\alpha \eta ^{2}-2T)-\beta
(2\eta -\alpha \eta ^{2}-2T)}\int_{0}^{T}(T-s)a(s)f(u(s))ds \\
&&-\int_{0}^{t}(t-s)a(s)f(u(s))ds.
\end{eqnarray*}
Denote
\begin{equation*}\label{eq-28}
K=\left\{u / u\in C([0,T],\mathbb{R}), u\geq0, \min_{t\in
[\eta,T]}u(t)\geq \gamma \|u\|\right\},
\end{equation*}
where $\gamma$ is defined in \eqref{eq-18}.
It is obvious that $K$
is a cone in $C([0,T],\mathbb{R})$. Moreover, from Lemma \ref{lem 2.2} and Lemma \ref{lem 2.4}, $AK\subset
K$. It is also easy to check that $A:K\rightarrow K$ is completely
continuous.

Superlinear case. $f_{0}=0$ and $f_{\infty}=\infty$.
Since $f_{0}=0$, we may choose $H_{1}>0$ so that $f(u)\leq \epsilon u$,
for $0<u\leq H_{1}$, where $\epsilon>0$ satisfies
\begin{equation*}\label{eq-29}
\epsilon
\frac{2(\beta+1)+T^{-1}\beta\eta(\alpha\eta+2)+\alpha\beta
T}{(2T-\alpha\eta^{2})-\beta(\alpha\eta^{2}-2\eta+2T)}\int_{0}^{T}T(T-s)a(s)ds\leq1.
\end{equation*}
Thus, if we let
\begin{equation*}\label{eq-30}
\Omega_{1}=\left\{u\in C([0,T],\mathbb{R}): \|u\|<H_{1}\right\},
\end{equation*}
then, for $u\in K\cap \partial\Omega_{1}$, we get
\begin{eqnarray*}
Au(t)&\leq&\frac{2\beta (1-\alpha \eta
)t-\beta (2T-\alpha \eta ^{2})}{(2T-\alpha\eta^{2})-\beta(\alpha\eta^{2}-2\eta+2T)}\int_{0}^{\eta }(\eta -s)a(s)f(u(s))ds \\
&&+\frac{\alpha (\beta -1)t-\alpha \beta \eta}{(2T-\alpha\eta^{2})-\beta(\alpha\eta^{2}-2\eta+2T)}\int_{0}^{\eta }(\eta -s)^{2}a(s)f(u(s))ds \\
&&+\frac{2\beta \eta-2(\beta-1)t }{(2T-\alpha\eta^{2})-\beta(\alpha\eta^{2}-2\eta+2T)}\int_{0}^{T}(T-s)a(s)f(u(s))ds \\
&\leq&\frac{2\beta T+\alpha\beta \eta ^{2}}{(2T-\alpha\eta^{2})-\beta(\alpha\eta^{2}-2\eta+2T)}\int_{0}^{\eta }(\eta -s)a(s)f(u(s))ds \\
&&+\frac{\alpha\beta T}{(2T-\alpha\eta^{2})-\beta(\alpha\eta^{2}-2\eta+2T)}\int_{0}^{\eta }(\eta -s)^{2}a(s)f(u(s))ds \\
&&+\frac{2\beta \eta+2T}{(2T-\alpha\eta^{2})-\beta(\alpha\eta^{2}-2\eta+2T)}\int_{0}^{T}(T-s)a(s)f(u(s))ds \\
&\leq&\frac{2T(\beta+1)+\beta\eta(\alpha\eta+2)}{(2T-\alpha\eta^{2})-\beta(\alpha\eta^{2}-2\eta+2T)}\int_{0}^{T}(T-s)a(s)f(u(s))ds \\
&&+\frac{\alpha\beta
T}{(2T-\alpha\eta^{2})-\beta(\alpha\eta^{2}-2\eta+2T)}\int_{0}^{\eta
}(\eta -s)^{2}a(s)f(u(s))ds\\
&\leq&\frac{2T(\beta+1)+\beta\eta(\alpha\eta+2)}{(2T-\alpha\eta^{2})-\beta(\alpha\eta^{2}-2\eta+2T)}\int_{0}^{T}(T-s)a(s)f(u(s))ds \\
&&+\frac{\alpha\beta T}{(2T-\alpha\eta^{2})-\beta(\alpha\eta^{2}-2\eta+2T)}\int_{0}^{T}T(T-s)a(s)f(u(s))ds \\
&=&\frac{2(\beta+1)+T^{-1}\beta\eta(\alpha\eta+2)+\alpha\beta
T}{(2T-\alpha\eta^{2})-\beta(\alpha\eta^{2}-2\eta+2T)}\int_{0}^{T}T(T-s)a(s)f(u(s))ds\\
&\leq&\epsilon \|u\|
\frac{2(\beta+1)+T^{-1}\beta\eta(\alpha\eta+2)+\alpha\beta
T}{(2T-\alpha\eta^{2})-\beta(\alpha\eta^{2}-2\eta+2T)}\int_{0}^{T}T(T-s)a(s)ds\\
&\leq&\|u\|.
\end{eqnarray*}
Thus $\|Au\|\leq\|u\|$, $u\in K\cap\partial\Omega_{1}$.

Further, since $f_{\infty}=\infty$, there exists $\widehat{H}_{2}>0$
such that $f(u)\geq\rho u$ for $u\geq \widehat{H}_{2}$, where
$\rho>0$ is chosen so that
\begin{equation*}\label{eq-31}
\rho \gamma
\frac{2\eta}{(2T-\alpha\eta^{2})-\beta(\alpha\eta^{2}-2\eta+2T)}\int_{\eta}^{T}(T-s)a(s)ds\geq1.
\end{equation*}
Let $H_{2}=\max\{2H_{1}, \frac{\widehat{H}_{2}}{\gamma}\}$ and
$\Omega_{2}=\left\{u\in C([0,T],\mathbb{R}): \|u\|<H_{2}\right\}$.
Then $u\in K\cap\partial\Omega_{2}$ implies that
\begin{equation*}\label{eq-32}
\min_{t\in [\eta,T]}u(t)\geq \gamma \|u\|=\gamma H_{2}\geq
\widehat{H}_{2},
\end{equation*}
and so,
\begin{eqnarray*}
Au(\eta)&=&\frac{2\beta (1-\alpha \eta
)\eta-\beta (2T-\alpha \eta ^{2})}{(2T-\alpha\eta^{2})-\beta(\alpha\eta^{2}-2\eta+2T)}\int_{0}^{\eta }(\eta -s)a(s)f(u(s))ds \\
&&+\frac{\alpha (\beta -1)\eta-\alpha \beta \eta}{(2T-\alpha\eta^{2})-\beta(\alpha\eta^{2}-2\eta+2T)}\int_{0}^{\eta }(\eta -s)^{2}a(s)f(u(s))ds \\
&&+\frac{2\beta \eta-2(\beta-1)\eta }{(2T-\alpha\eta^{2})-\beta(\alpha\eta^{2}-2\eta+2T)}\int_{0}^{T}(T-s)a(s)f(u(s))ds \\
&&-\int_{0}^{\eta}(\eta-s)a(s)f(u(s))ds\\
&=&\frac{2\eta}{(2T-\alpha\eta^{2})-\beta(\alpha\eta^{2}-2\eta+2T)}\int_{0}^{T}(T-s)a(s)f(u(s))ds \\
&&-\frac{\alpha \eta}{(2T-\alpha\eta^{2})-\beta(\alpha\eta^{2}-2\eta+2T)}\int_{0}^{\eta }(\eta^{2}-2\eta s+s^{2})a(s)f(u(s))ds \\
&&-\frac{2T-\alpha
\eta^{2}}{(2T-\alpha\eta^{2})-\beta(\alpha\eta^{2}-2\eta+2T)}\int_{0}^{\eta}(\eta
-s)a(s)f(u(s))ds\\
&=&\frac{2\eta}{(2T-\alpha\eta^{2})-\beta(\alpha\eta^{2}-2\eta+2T)}\int_{0}^{T}(T-s)a(s)f(u(s))ds \\
&&+\frac{\alpha\eta^{2}}{(2T-\alpha\eta^{2})-\beta(\alpha\eta^{2}-2\eta+2T)}\int_{0}^{\eta}sa(s)f(u(s))ds \\
&&+\frac{2T}{(2T-\alpha\eta^{2})-\beta(\alpha\eta^{2}-2\eta+2T)}\int_{0}^{\eta}sa(s)f(u(s))ds \\
&&-\frac{\alpha\eta}{(2T-\alpha\eta^{2})-\beta(\alpha\eta^{2}-2\eta+2T)}\int_{0}^{\eta}s^{2}a(s)f(u(s))ds \\
&&-\frac{2T\eta}{(2T-\alpha\eta^{2})-\beta(\alpha\eta^{2}-2\eta+2T)}\int_{0}^{\eta}a(s)f(u(s))ds\\
&=&\frac{2\eta}{(2T-\alpha\eta^{2})-\beta(\alpha\eta^{2}-2\eta+2T)}\int_{\eta}^{T}(T-s)a(s)f(u(s))ds \\
&&+\frac{2(T-\eta)}{(2T-\alpha\eta^{2})-\beta(\alpha\eta^{2}-2\eta+2T)}\int_{0}^{\eta}sa(s)f(u(s))ds \\
&&+\frac{\alpha\eta}{(2T-\alpha\eta^{2})-\beta(\alpha\eta^{2}-2\eta+2T)}\int_{0}^{\eta}s(\eta-s)a(s)f(u(s))ds\\
&\geq&\frac{2\eta}{(2T-\alpha\eta^{2})-\beta(\alpha\eta^{2}-2\eta+2T)}\int_{\eta}^{T}(T-s)a(s)f(u(s))ds \\
&\geq&\frac{2\eta\rho}{(2T-\alpha\eta^{2})-\beta(\alpha\eta^{2}-2\eta+2T)}\int_{\eta}^{T}(T-s)a(s)u(s)ds \\
&\geq&\frac{2\eta\rho\gamma \|u\|}{(2T-\alpha\eta^{2})-\beta(\alpha\eta^{2}-2\eta+2T)}\int_{\eta}^{T}(T-s)a(s)ds \\
&\geq&\|u\|.
\end{eqnarray*}
Hence, $\|Au\|\geq\|u\|$, $u\in K\cap\partial\Omega_{2}$. By the
first part of Theorem \ref{theo 1.1}, $A$ has a fixed point in
$K\cap (\overline{\Omega}_{2}\backslash \Omega _{1})$ such that
$H_{1}\leq\|u\|\leq H_{2}$. This completes the superlinear part of
the theorem.

Sublinear case. $f_{0}=\infty$ and $f_{\infty}=0$. Since
$f_{0}=\infty$, choose $H_{3}>0$ such that $f(u)\geq Mu$ for
$0<u\leq H_{3}$, where $M>0$ satisfies
\begin{equation*}\label{eq-33}
M\gamma\frac{2\eta}{(2T-\alpha\eta^{2})-\beta(\alpha\eta^{2}-2\eta+2T)}\int_{\eta}^{T}(T-s)a(s)ds\geq1.
\end{equation*}
Let $\Omega_{3}=\left\{u\in C([0,T],\mathbb{R}):
\|u\|<H_{3}\right\}$, then for $u\in K\cap\partial\Omega_{3}$, we
get
\begin{eqnarray*}
Au(\eta)&=&\frac{2\eta}{(2T-\alpha\eta^{2})-\beta(\alpha\eta^{2}-2\eta+2T)}\int_{0}^{T}(T-s)a(s)f(u(s))ds \\
&&-\frac{\alpha \eta}{(2T-\alpha\eta^{2})-\beta(\alpha\eta^{2}-2\eta+2T)}\int_{0}^{\eta }(\eta-s)^{2}a(s)f(u(s))ds \\
&&-\frac{2T-\alpha
\eta^{2}}{(2T-\alpha\eta^{2})-\beta(\alpha\eta^{2}-2\eta+2T)}\int_{0}^{\eta}(\eta
-s)a(s)f(u(s))ds\\
&\geq&\frac{2\eta}{(2T-\alpha\eta^{2})-\beta(\alpha\eta^{2}-2\eta+2T)}\int_{\eta}^{T}(T-s)a(s)f(u(s))ds\\
&\geq&\frac{2\eta M}{(2T-\alpha\eta^{2})-\beta(\alpha\eta^{2}-2\eta+2T)}\int_{\eta}^{T}(T-s)a(s)u(s)ds\\
&\geq&M\gamma\frac{2\eta
\|u\|}{(2T-\alpha\eta^{2})-\beta(\alpha\eta^{2}-2\eta+2T)}\int_{\eta}^{T}(T-s)a(s)ds\\
&\geq&\|u\|.
\end{eqnarray*}
Thus $\|Au\|\geq\|u\|$, $u\in K\cap\partial\Omega_{3}$.
Now, since $f_{\infty}=0$, there exists $\widehat{H}_{4}>0$ so that $f(u)\leq
\lambda u$ for $u\geq \widehat{H}_{4}$, where $\lambda>0$ satisfies
\begin{equation*}\label{eq-34}
\lambda
\frac{2(\beta+1)+T^{-1}\beta\eta(\alpha\eta+2)+\alpha\beta
T}{(2T-\alpha\eta^{2})-\beta(\alpha\eta^{2}-2\eta+2T)}\int_{0}^{T}T(T-s)a(s)ds\leq1.
\end{equation*}
We consider two cases:

Case (i). Suppose $f$ is bounded, say
$f(u)\leq N$ for all $u\in[0,\infty)$. Choosing
$H_{4}\geq\max\{2H_{3},\frac{N}{\lambda}\}$.
For $u\in K$ with $\|u\|=H_{4}$, we have
\begin{eqnarray*}
Au(t)&=&\frac{2\beta (1-\alpha \eta
)t-\beta (2T-\alpha \eta ^{2})}{(2T-\alpha\eta^{2})-\beta(\alpha\eta^{2}-2\eta+2T)}\int_{0}^{\eta }(\eta -s)a(s)f(u(s))ds \\
&&+\frac{\alpha (\beta -1)t-\alpha \beta \eta}{(2T-\alpha\eta^{2})-\beta(\alpha\eta^{2}-2\eta+2T)}\int_{0}^{\eta }(\eta -s)^{2}a(s)f(u(s))ds \\
&&+\frac{2\beta \eta-2(\beta-1)t }{(2T-\alpha\eta^{2})-\beta(\alpha\eta^{2}-2\eta+2T)}\int_{0}^{T}(T-s)a(s)f(u(s))ds \\
&&-\int_{0}^{t}(t-s)a(s)f(u(s))ds\\
&\leq&
\frac{2(\beta+1)+T^{-1}\beta\eta(\alpha\eta+2)+\alpha\beta
T}{(2T-\alpha\eta^{2})-\beta(\alpha\eta^{2}-2\eta+2T)}\int_{0}^{T}T(T-s)a(s)f(u(s))ds\\
&\leq& N
\frac{2(\beta+1)+T^{-1}\beta\eta(\alpha\eta+2)+\alpha\beta
T}{(2T-\alpha\eta^{2})-\beta(\alpha\eta^{2}-2\eta+2T)}\int_{0}^{T}T(T-s)a(s)ds\\
&\leq& H_{4}\lambda
\frac{2(\beta+1)+T^{-1}\beta\eta(\alpha\eta+2)+\alpha\beta
T}{(2T-\alpha\eta^{2})-\beta(\alpha\eta^{2}-2\eta+2T)}\int_{0}^{T}T(T-s)a(s)ds\\
&\leq& H_{4},
\end{eqnarray*}
and therefore $\|Au\|\leq\|u\|$.
\end{proof}

Case (ii). If $f$ is unbounded, then we know from $f\in
C([0,\infty),[0,\infty))$ that there is $H_{4}$:
$H_{4}\geq\max\{2H_{3},\frac{\widehat{H}_{4}}{\gamma}\}$ such that
\begin{equation*}\label{eq-35}
f(u)\leq f(H_{4})\ \text{for}\ u\in[0,H_{4}].
\end{equation*}
Then for $u\in K$ and $\|u\|=H_{4}$, we have
\begin{eqnarray*}
Au(t)&\leq&
\frac{2(\beta+1)+T^{-1}\beta\eta(\alpha\eta+2)+\alpha\beta
T}{(2T-\alpha\eta^{2})-\beta(\alpha\eta^{2}-2\eta+2T)}\int_{0}^{T}T(T-s)a(s)f(u(s))ds\\
&\leq& \frac{2(\beta+1)+T^{-1}\beta\eta(\alpha\eta+2)+\alpha\beta
T}{(2T-\alpha\eta^{2})-\beta(\alpha\eta^{2}-2\eta+2T)}\int_{0}^{T}T(T-s)a(s)f(H_{4}))ds\\
&\leq& H_{4} \lambda
\frac{2(\beta+1)+T^{-1}\beta\eta(\alpha\eta+2)+\alpha\beta
T}{(2T-\alpha\eta^{2})-\beta(\alpha\eta^{2}-2\eta+2T)}\int_{0}^{T}T(T-s)a(s)ds\\
&\leq& H_{4}=\|u\|.
\end{eqnarray*}
Therefore, in either case we may set
\begin{equation*}\label{eq-36}
\Omega_{4}=\left\{u\in C([0,T],\mathbb{R}): \|u\|<H_{4}\right\},
\end{equation*}
and for $u\in K\cap\partial\Omega_{4}$ we may have $\|Au\|\leq
\|u\|$. By the second part of Theorem \ref{theo 1.1}, it follows
that $A$ has a fixed point in $K\cap
(\overline{\Omega}_{4}\backslash \Omega _{3})$ such that
$H_{3}\leq\|u\|\leq H_{4}$. This completes the sublinear part of
the theorem. Therefore, the problem \eqref{eq-4}-\eqref{eq-5} has
at least one positive solution.

\section{Some examples}
In this section, in order to illustrate our result, we consider some examples.

\begin{exmp}
Consider the boundary value problem

\begin{equation}\label{eq-37}
{u^{\prime \prime }}(t)+tu^{p}=0, \  \ 0<t<2,
\end{equation}

\begin{equation}\label{eq-38}
u(0)=\frac{1}{2}u(\frac{3}{2}), \  \ u(2)= \int_{0}^{\frac{3}{2}}u(s)ds.
\end{equation}

Set $\beta=1/2$, $\alpha=1$, $\eta=3/2$, $T=2$, $a(t)=t$, $f(u)=u^{p}$. We can show that $0<\alpha=1<16/9=2T/{\eta^{2}}$, $0<\beta=1/2<7/13=(2T-\alpha \eta ^{2})/(\alpha \eta ^{2}-2\eta
+2T)$.\\
Now we consider the existence of positive solutions of the problem \eqref{eq-37}, \eqref{eq-38} in two
cases.\\
Case 1: $p>1$. In this case, $f_{0}=0$, $f_{\infty}=\infty$ and (i) holds. Then \eqref{eq-37}, \eqref{eq-38} has at least one positive solution.\\
Case 2: $p\in (0, 1)$. In this case, $f_{0}=\infty$, $f_{\infty}=0$ and (ii) holds. Then (\eqref{eq-37}, \eqref{eq-38} has at least one positive solution.
\end{exmp}

\begin{exmp}
Consider the boundary value problem

\begin{equation}\label{eq-39}
{u^{\prime \prime }}(t)+t^{2}u^{2}\ln{(1+u)}=0, \  \ 0<t<\frac{3}{4},
\end{equation}

\begin{equation}\label{eq-40}
u(0)=\frac{1}{10}u(\frac{1}{4}), \  \ u(\frac{3}{4})= 20\int_{0}^{\frac{1}{4}}u(s)ds.
\end{equation}

Set $\beta=1/10$, $\alpha=20$, $\eta=1/4$, $T=3/4$, $a(t)=t^{2}$, $f(u)=u^{2}\ln{(1+u)}$. We can show that $0<\alpha=20<24=2T/{\eta^{2}}$, $0<\beta=1/10<1/9=(2T-\alpha \eta ^{2})/(\alpha \eta ^{2}-2\eta
+2T)$.
Through a simple calculation we can get $f_{0}=0$ and $f_{\infty}=\infty$. Thus, by the
first part of Theorem \ref{theo 3.1}, we can get that the problem \eqref{eq-39}, \eqref{eq-40}) has at least one positive solution.
\end{exmp}

\begin{exmp}
Consider the boundary value problem

\begin{equation}\label{eq-41}
{u^{\prime \prime }}(t)+e^{t}\frac{\sin{u}+\ln{(1+u)}}{u^{2}}=0, \  \ 0<t<1,
\end{equation}

\begin{equation}\label{eq-42}
u(0)=u(\frac{1}{3}), \  \ u(1)= 2\int_{0}^{\frac{1}{3}}u(s)ds.
\end{equation}

Set $\beta=1$, $\alpha=2$, $\eta=1/3$, $T=1$, $a(t)=e^{t}$, $f(u)=(\sin{u}+\ln{(1+u)})/{u^{2}}$. We can show that $0<\alpha=2<18=2T/{\eta^{2}}$, $0<\beta=1<8/7=(2T-\alpha \eta ^{2})/(\alpha \eta ^{2}-2\eta
+2T)$.
Through a simple calculation we can get $f_{0}=\infty$ and $f_{\infty}=0$. Thus, by the
second part of Theorem \ref{theo 3.1}, we can get that the problem \eqref{eq-41}, \eqref{eq-42}) has at least one positive solution.
\end{exmp}

\end{document}